\documentclass[final,11pt,sort&compress]{elsarticle}

\usepackage{verbatim,a4wide}

\usepackage{amsmath,bm}
\usepackage{amssymb}
\usepackage{amsthm}
\usepackage[a4paper,left=2.5cm,right=2.5cm,top=2cm,bottom=3cm]{geometry}
\usepackage{graphicx}
\usepackage{caption}
\usepackage{subfig}
\usepackage{grffile}
\usepackage{float}
\usepackage{bbm}
\usepackage{indentfirst}
\usepackage{booktabs}
\usepackage{multirow}
\usepackage{rotating}
\usepackage{hyperref}
\usepackage{epstopdf}

\usepackage[cp1250]{inputenc}
\usepackage[OT4]{fontenc}
\usepackage[english]{babel}

\numberwithin{equation}{section}

\newtheorem{thm}{Theorem}[section]
\newtheorem{lem}[thm]{Lemma}
\newtheorem{fact}[thm]{Fact}

\newtheorem{alg}[thm]{Algorithm}
\newtheorem{prob}[thm]{Problem}

\newtheorem{remark}[thm]{Remark}

\newcommand{\ra}[1]{\renewcommand{\arraystretch}{#1}}

\makeatletter
\def\ps@pprintTitle{%
 \let\@oddhead\@empty
 \let\@evenhead\@empty
 \def\@oddfoot{\centerline{\thepage}}%
 \let\@evenfoot\@oddfoot}
\makeatother

\begin{document}

\begin{frontmatter}

\title{Efficient degree reduction of B\'ezier curves with box constraints using dual bases}

\author{Przemys{\l}aw Gospodarczyk\corref{cor}}
\ead{pgo@ii.uni.wroc.pl}
\author{Pawe{\l} Wo\'{z}ny}
\ead{Pawel.Wozny@ii.uni.wroc.pl}
\cortext[cor]{Corresponding author. Fax {+}48 71 3757801}
\address{Institute of Computer Science, University of Wroc{\l}aw,
         ul.~Joliot-Curie 15, 50-383 Wroc{\l}aw, Poland}

\begin{abstract}
In this paper, we give an efficient algorithm of degree reduction of B\'ezier curves with box constraints.
The idea is to combine the previous iterative approach, that has been presented recently in
(P. Gospodarczyk, Comput.~Aided Des.~62 (2015), 143--151), with a fast method of construction of dual bases from
(P. Wo\'zny, J.~Comput.~Appl.~Math.~260 (2014), 301--311) and a new efficient method of modification of dual bases.
\end{abstract}

\begin{keyword}
B\'{e}zier curves, degree reduction, box constraints, restricted area, constrained least squares approximation, dual bases.
\end{keyword}

\end{frontmatter}

\section{Introduction}\label{Sec:Intro}

Degree reduction of B\'ezier curves is to approximate an original B\'ezier curve of a certain degree with a different one of a lower degree.
 Such a procedure can be used in \textit{data exchange}, \textit{data compression} and \textit{data comparison}.
Therefore, degree reduction of B\'ezier curves is an important problem in CAGD.
A \textit{conventional} approach to this problem is to minimize a chosen error function subject to \textit{parametric} or \textit{geometric} continuity constraints at the endpoints. In the past $30$ years, the conventional approach has been extensively studied (see, e.g., \cite{Ahn03,ALPY04,HB16,B06,BWX95,Eck95,GLW15,GLW17,Lac91,Lu13,Lu15,LW06,LW08,LX16,LPR99,Sun05,SL04,WL09} and the lists of references given there). Some of the algorithms of conventional degree reduction are based on properties of the dual basis for the Bernstein basis, i.e., the so-called \textit{dual Bernstein polynomials}. Consequently, they have the lowest computational complexity among all existing methods, and they avoid matrix inversion. For details, see \cite{GLW15,GLW17,Lu15,WL09}. Unfortunately, the conventional degree reduction has its serious drawbacks. As a result of this common strategy, one may obtain control points that are located far away from the plot of the curve. Consequently, further editing of the resulting curve may be difficult. In addition, the \textit{convex hull property} of the curve is useless in some practical applications. For details,
see \cite[Section 2]{Gos15}.

In \cite{Gos15}, the goal was to eliminate the mentioned issues that arise in the case of the conventional approach.
To do so, one of us formulated and solved a new problem of degree reduction of B\'ezier curves with box constraints
(for an analogous problem of merging of B\'ezier curves with box constraints, see \cite{GW14}).
Now, we recall that problem. Let $\Pi_n^2$ denote the space of all parametric polynomials in $\mathbb{R}^2$ of degree at most $n$.
\begin{prob}\,[\textsf{Degree reduction of B\'ezier curves with box constraints}]\label{P:Problem}\\
Let there be given a \emph{B\'{e}zier curve} $P_n \in \Pi_n^2$,
\begin{equation*}
    P_n(t) := \sum_{i = 0}^{n} p_{i}B_i^{n}(t) \qquad (0 \leq t \leq 1),
\end{equation*}
where $p_i := \left(p_i^x,p_i^y\right) \in \mathbb{R}^2$ are called \emph{control points}, and $n$ is the \emph{degree} of the curve. Here
$$
B_i^n(t) := \binom{n}{i}t^i(1-t)^{n-i} \qquad (0 \leq i \leq n)
$$
are \emph{Bernstein polynomials of degree $n$}. Find a B\'{e}zier curve $R_m \in \Pi_m^2$ of a lower degree $m$,
\begin{equation*}
    R_m(t) := \sum_{i = 0}^{m} r_{i}B_i^{m}(t) \qquad (0 \leq t \leq 1;\ m<n),
\end{equation*}
satisfying the following conditions:
\begin{itemize}
    \item[(i)] value of the \emph{least squares error}
    \begin{equation}\label{EQ:min}
         E \equiv \|P_n-R_m\|_2^T := \sqrt{\sum_{k=0}^{N}\| P_n(t_k) - R_m(t_k)\|^2}
    \end{equation}
    is minimized, where $T := \{t_k\}^N_{k=0}$ $(N \in \mathbb{N})$ is a given strictly increasing sequence whose elements are in the interval $[0,\,1]$, and
    $\|\cdot\|$ is the Euclidean vector norm in $\mathbb{R}^2$;

    \item[(ii)] $P_n$ and $R_m$ are $C^{\alpha,\beta}$-\emph{continuous} ($\alpha, \beta \geq -1$ and $\alpha+\beta < m-1$) at the endpoints, i.e.,
        \begin{equation}\label{EQ:cont}
        \left.\begin{array}{ll}
        \displaystyle P_n^{(i)}(0) = R_m^{(i)}(0) &\qquad (i = 0, 1,\ldots, \alpha),\\[1ex]
        \displaystyle P_n^{(j)}(1) = R_m^{(j)}(1) &\qquad (j = 0, 1,\ldots, \beta);
        \end{array}\right\}
        \end{equation}

    \item[(iii)] \emph{inner} control points $r_i := \left(r_i^x,r_i^y\right)$ $(\alpha < i < m-\beta)$ are located inside the specified rectangular area, including edges of the rectangle, i.e., the following \emph{box constraints} are fulfilled:
     \begin{equation}\label{EQ:box}
         l_z \leq r_{i}^z \leq u_z \qquad (i = \alpha+1, \alpha+2,\ldots, m-\beta-1;\ z= x, y),
     \end{equation}
     where $l_x, l_y, u_x, u_y \in \mathbb R$.
\end{itemize}
\end{prob}

\begin{remark}\label{R:Trad}\normalfont
Further on in the paper, the minimization of~\eqref{EQ:min}, with the conditions~\eqref{EQ:cont}, but without the box constraints~\eqref{EQ:box} is called the \textit{traditional degree reduction}.
\end{remark}

 Observe that the key idea was to impose the box constraints~\eqref{EQ:box}. These additional restrictions guarantee that the resulting inner control points will not go outside the specified rectangular area. However, because of the new restrictions, Problem~\ref{P:Problem} is much more difficult to solve than the conventional problems of degree reduction, and the mentioned well-developed algorithms based on dual Bernstein polynomials cannot be applied.

In \cite[Section 5]{Gos15}, one can find step by step instructions on how to solve Problem~\ref{P:Problem} using \textit{BVLS algorithm (bounded-variable least-squares)} \cite{SP95}. This iterative \textit{active-set method} is modelled on \textit{NNLS algorithm (non-negative least-squares)} \cite{LH74}.
The main computational cost of a single iteration of the algorithm is associated with solving the so-called \textit{subproblem}
(see \cite[Section 5, Step 3]{Gos15}). Consequently, further research can focus on finding fast methods of solving the subproblem.

Given the high quality of methods based on dual bases that are used for the conventional degree reduction, the main goal of this paper is to introduce a new method of a similar quality for the subproblem. First, using an observation that the subproblems in consecutive iterations are \textit{quite similar}, we establish a connection between them (see Section~\ref{Sec:Idea}). This leads us to an idea of using two types of fast connections between certain dual functions. As it turns out, the connections of the first type were already given by one of us in \cite{Woz13,Woz14} (for the sake of completeness, we recall them in Section~\ref{Sec:Dual}), whereas the connections of the second type are new (see Section~\ref{Sec:Prop}). Thanks to dual bases,
our new method of solving Problem~\ref{P:Problem} avoids dealing with a system of normal equations in every iteration.
As a result, it is faster than the one from \cite{Gos15}, and it avoids matrix inversion which is considered to be risky from the numerical point of view (see, e.g., \cite[Section 14]{Hig02}). Examples showing the efficiency of the new approach are presented in Section~\ref{Sec:Ex}. Section~\ref{Sec:Conc} concludes the paper.

Further on in this section, we give a short introduction to dual bases since they are our main tool in solving Problem~\ref{P:Problem}.

\subsection{A short introduction to dual bases}\label{Sec:DIntro}

A \textit{dual basis} $D_n := \left\{d_0^{(n)},d_1^{(n)},\ldots,d_n^{(n)}\right\}$ for a basis $B_n := \left\{b_0,b_1,\ldots,b_n\right\}$ of the linear space $\mathcal{B}_n:= \mbox{span}\,B_n$ satisfies the following conditions:
\begin{equation*}
\left\{\begin{array}{ll}
\displaystyle \mbox{span}\,D_n=\mathcal{B}_n, &\\[1ex]
\displaystyle \left<b_i,d_j^{(n)}\right>=\delta_{ij} & (i,j = 0,1,\ldots,n),
\end{array}\right.
\end{equation*}
where $\delta_{ij}$ equals $1$ if $i = j$, and $0$ otherwise; $\left<\cdot,\cdot\right>\,:\,\mathcal{B}_n\times \mathcal{B}_n\rightarrow\mathbb{R}$ is an inner product; and $d_j^{(n)}$ $(j=0,1,\ldots,n)$ are called \textit{dual functions}.

Now, we present some well-known facts about dual bases (see, e.g., \cite[Section 1]{Woz13}).
\begin{fact}\label{T:DualRepr}
Every $f_n \in \mathcal{B}_n$ can be written in the following way:
$$
f_n = \sum_{i=0}^n \left<f_n,d_i^{(n)}\right>b_i.
$$
\end{fact}

\begin{fact}\label{T:DualLS}
Given a function $g$,
\begin{equation*}
f_n^\ast = \sum_{i=0}^n \left<g,d_i^{(n)}\right>b_i
\end{equation*}
is the best least squares approximation of $g$ in the space $\mathcal{B}_n$, i.e.,
$$
\|g-f_n^\ast\|_2=\min_{f_n\in\mathcal{B}_n}\|g-f_n\|_2,
$$
where $\|\cdot\|_2:=\sqrt{\left<\cdot,\cdot\right>}$ denotes the least squares norm.
\end{fact}

Recently, dual bases with their applications in numerical analysis and CAGD have been extensively studied.
For example, dual Bernstein polynomials have found their application in several important algorithms associated with B\'ezier curves
(see, e.g., \cite{GLW15,GLW17,Lu15,WGL15,WL09}). For similar algorithms concerning B\'ezier surfaces, see \cite{LW11b} and the list of references given there. In \cite{GW16,Woz14}, one can find some results on \textit{dual B-spline functions}.  Dual basis functions in subspaces of inner product spaces were discussed in \cite{Ker13}. As for properties of dual bases in general, see \cite{Woz13,Woz14}.

\section{Solving the subproblems -- an idea for the improvement}\label{Sec:Idea}

In this section, we recall some basic facts about Problem~\ref{P:Problem} and the iterative algorithm of solving it (for a more detailed description of the full algorithm, see \cite[Section 5 and Appendix]{Gos15}). We also establish a connection between consecutive iterations of the algorithm, which leads us
to an idea for the improvement of the method.

First, recall that the continuity conditions~\eqref{EQ:cont} imply the well-known formulas \cite[(3.4)]{Gos15} and \cite[(3.5)]{Gos15} for the control points $r_0,r_1,\ldots,r_\alpha$ and $r_{m-\beta},r_{m-\beta+1},\ldots,r_{m}$, respectively. Moreover, Problem~\ref{P:Problem} can be solved in a componentwise way. Therefore, it is sufficient to explain how to compute $r^x_{\alpha+1},r^x_{\alpha+2},\ldots,r^x_{m-\beta-1}$ (see \cite[Remark 3.3]{Gos15}).

Now, we recall some rules of the iterative algorithm from \cite{Gos15}, and focus mainly on the subproblem which is an essential part of every iteration of that algorithm.
We start with some definitions. We define a set $\mathcal{C} := \left\{0,1,\ldots,\alpha,m-\beta,m-\beta+1,\ldots,m\right\}$. Notice that it contains indices of variables which are already computed thanks to the continuity conditions~\eqref{EQ:cont}. Let $\mathcal{F}_i$ be a set of indices of variables satisfying strict version of inequalities~\eqref{EQ:box} before solving the subproblem in the $i$th iteration. Furthermore, we assume that sets $\mathcal{L}_i$ and $\mathcal{U}_i$ contain indices of variables which, before solving the subproblem in the $i$th iteration, have the minimum and maximum permissible value, respectively (cf.~\eqref{EQ:box}). In the $i$th iteration, the subproblem is formulated as follows.
\begin{prob}\,[\textsf{$i$th subproblem}]\label{P:SubProblem}\\
Find the optimal values of variables (i.e., coordinates of the control points) whose indices are in set $\mathcal{F}_i$ subject to the fixed values of variables whose indices are in set $\mathcal{C} \cup \mathcal{L}_i \cup \mathcal{U}_i$. More precisely, we look for
$\psi_i^{\ast} \in \Pi_m^{\mathcal{F}_i}$ such that:
\begin{equation}\label{EQ:SubProb}
\|\varphi_i - \psi_i^{\ast}\|_2^T = \min_{\psi_i\in\Pi_m^{\mathcal{F}_i}}\|\varphi_i - \psi_i\|_2^T
\end{equation}
(cf.~\eqref{EQ:min}), where
\begin{align*}
&\Pi_m^{\mathcal{A}} := \textnormal{span}\left\{B_j^m \::\: j \in \mathcal{A} \right\},\\[1ex]
&\varphi_i := \sum_{h = 0}^{n} p_{h}^xB_h^{n} - \sum_{j \in \mathcal{C}}r_j^xB_j^m - \sum_{j \in \mathcal{L}_i}l_xB_j^m - \sum_{j \in \mathcal{U}_i}u_xB_j^m.
\end{align*}
\end{prob}

\begin{remark}\normalfont
As we shall see, the consecutive subproblems are \textit{quite similar} and related. We did not use that fact in the previous article \cite{Gos15},
where the \textit{older method} deals with a system of normal equations in every iteration of the algorithm in order to solve the subproblems.
In this paper, our goal is to avoid this inefficient approach which additionally can be associated with matrix inversion (cf.~\cite[(A.3)]{Gos15}).
Recall that matrix inversion is considered to be risky from the numerical point of view (see, e.g., \cite[Section 14]{Hig02}).
\end{remark}

\noindent Now, let us analyze the consecutive subproblems.

\textbf{The first subproblem.} At the beginning of the algorithm, we set
$$
\mathcal{L}_1 = \mathcal{U}_1 := \emptyset, \qquad \mathcal{F}_1 := \left\{\alpha+1, \alpha+2,\ldots, m-\beta-1\right\}
$$
(see \cite[Section 5, Step 1]{Gos15}). Therefore, we have
$$
\varphi_1 := \sum_{h = 0}^{n} p_{h}^xB_h^{n} - \sum_{j \in \mathcal{C}}r_j^xB_j^m,
$$
and the first subproblem is to compute the optimal element  $\psi_1^{\ast} \in \Pi_m^{\mathcal{F}_1}$
written in the basis
\begin{equation}\label{E:First}
\left\{B_j^m \::\: j \in \mathcal{F}_{1}\right\} = \left\{B^m_{\alpha+1}, B^m_{\alpha+2},\ldots, B^m_{m-\beta-1}\right\}.
\end{equation}

Next, let us consider the $i$th subproblem $(i > 1)$ and its relation with the previous one.
There are two possibilities.

\textbf{Case 1 of the $\bm{i}$th subproblem $\bm{(i > 1)}$.} One element $q$ was transferred from $\mathcal{L}_{i-1}$ or $\mathcal{U}_{i-1}$ to $\mathcal{F}_{i-1}$ (see \cite[Section 5, Step 2]{Gos15}). Therefore, we set $\mathcal{F}_{i} := \mathcal{F}_{i-1} \cup \left\{q\right\}$ and (($\mathcal{L}_{i} := \mathcal{L}_{i-1} \setminus \left\{q\right\}\,\wedge\,\mathcal{U}_{i} := \mathcal{U}_{i-1}$) or ($\mathcal{U}_{i} := \mathcal{U}_{i-1} \setminus \left\{q\right\}\,\wedge\,\mathcal{L}_{i} := \mathcal{L}_{i-1}$)).
According to~\eqref{EQ:SubProb}, for the given
\begin{equation}\label{E:phi1}
\varphi_i := \varphi_{i-1} + sB^m_q \qquad (s = l_x \;\;\mbox{or}\;\; s = u_x),
\end{equation}
we look for the optimal element $\psi_i^{\ast} \in \Pi_m^{\mathcal{F}_i}$ written in the basis $\{B_j^m \::\: j \in \mathcal{F}_{i}\}$. Notice that this basis is related with the one from the previous iteration, i.e.,
\begin{equation}\label{E:Case1}
\left\{B_j^m \::\: j \in \mathcal{F}_{i} \right\} = \left\{B_h^m \::\: h \in \mathcal{F}_{i-1} \right\} \cup \left\{B_q^m\right\}.
\end{equation}

\textbf{Case 2 of the $\bm{i}$th subproblem $\bm{(i > 1)}$.} At least one element was transferred from $\mathcal{F}_{i-1}$ to $\mathcal{L}_{i-1}$ or $\mathcal{U}_{i-1}$ (see \cite[Section 5, Step 5]{Gos15}). Apart from very rare cases, exactly one element $q$ was transferred. If the rare case occurred, then the procedure that transfers one element should be applied repeatedly. We set $\mathcal{F}_{i} := \mathcal{F}_{i-1} \setminus \left\{q\right\}$ and (($\mathcal{L}_{i} := \mathcal{L}_{i-1} \cup \left\{q\right\}\,\wedge\,\mathcal{U}_{i} := \mathcal{U}_{i-1}$) or ($\mathcal{U}_{i} := \mathcal{U}_{i-1} \cup \left\{q\right\}\,\wedge\,\mathcal{L}_{i} := \mathcal{L}_{i-1}$)). This time, we have
\begin{equation}\label{E:phi2}
\varphi_i := \varphi_{i-1} - sB^m_q \qquad (s = l_x \;\;\mbox{or}\;\; s = u_x)
\end{equation}
and, according to~\eqref{EQ:SubProb}, we look for the optimal element $\psi_i^{\ast} \in
\Pi_m^{\mathcal{F}_i}$ written in the basis $\{B_j^m \::\: j \in \mathcal{F}_{i}\}$.
One can see clearly that this basis is related with the previous one, i.e.,
\begin{equation}\label{E:Case2}
\left\{B_j^m \::\: j \in \mathcal{F}_{i} \right\} = \left\{B_h^m \::\: h \in \mathcal{F}_{i-1} \right\} \setminus \left\{B_q^m\right\}
\end{equation}
(cf.~\eqref{E:Case1}).

\textbf{The goal.} According to Fact~\ref{T:DualLS} (cf.~Problem~\ref{P:SubProblem}), dual bases are useful in solving least squares problems such as the subproblems that we are dealing with in each iteration. Our goal is to solve every subproblem using dual bases. In order to do so, we must have the dual bases for the bases~\eqref{E:First}, \eqref{E:Case1} and \eqref{E:Case2} with respect to the inner product
$$
\langle f,g\rangle_T := \sum_{k=0}^Nf(t_k)g(t_k)
$$
(cf.~\eqref{EQ:min}), i.e., the dual basis must be updated in each iteration. Then, for each subproblem, we can use Fact~\ref{T:DualLS} to get its optimal solution $\psi_i^{\ast}$.\\

\noindent\textbf{The plan.}
\vspace{-2mm}
\begin{itemize}
\item[\textbf{I.}] In the first iteration, the dual basis for the basis~\eqref{E:First} must be computed from scratch. We recall such an algorithm in Section~\ref{Sec:Dual} (see Algorithm~\ref{A:DN}).
\item[\textbf{II.}]In the $i$th iteration $(i > 1)$, we must find a way to update the dual basis from the previous iteration in an efficient way. As a result, we will obtain the current dual basis for the $i$th iteration. There are two cases to solve.
\begin{itemize}
\item[\textbf{1.}] In case 1, one basis element is being added (see~\eqref{E:Case1}). Therefore, having the dual basis from the previous iteration, our algorithm should compute the dual basis for the expanded basis~\eqref{E:Case1}. An algorithm of such a type was already given in \cite{Woz13,Woz14}. We recall it in Section~\ref{Sec:Dual} (see Algorithm~\ref{A:DnDn1}).
\item[\textbf{2.}] In case 2, we are dealing with the reverse problem, i.e., one basis element is being removed (see~\eqref{E:Case2}). More precisely, having the dual basis from the previous iteration, our algorithm should compute the dual basis for the reduced basis~\eqref{E:Case2}. This is a new problem and its solution has never been published before. We give it in Section~\ref{Sec:Prop} (see Algorithm~\ref{A:Dn1Dn}).
\end{itemize}
\end{itemize}

\section{Construction of dual bases -- earlier work}\label{Sec:Dual}

Suppose that $B_n$ is the given basis of the space $\mathcal{B}_n$ and the dual basis $D_{n}$ with respect
to the inner product $\left<\cdot,\cdot\right>$ is known as well (here we use the notation from Section~\ref{Sec:DIntro}).
In \cite{Woz14}, one of us proposed an efficient method of constructing the dual basis
$$
D_{n+1} := \left\{d_0^{(n+1)},d_1^{(n+1)},\ldots,d_{n+1}^{(n+1)}\right\}
$$
for $B_{n+1} := B_n \cup \left\{b_{n+1}\right\}$. See also the previous method given in \cite{Woz13}.
Further on in this section, we recall the connection between $D_{n}$ and $D_{n+1}$ as well as the algorithm of constructing $D_{n+1}$.
Notice that this is exactly the algorithm that is needed in case 1 of the $i$th subproblem $(i > 1)$ (see Section~\ref{Sec:Idea}).

\begin{thm}[\cite{Woz14}]\label{T:RelDnDn1}
The dual functions from $D_n$ and $D_{n+1}$ are related in the following way:
\begin{equation}\label{EQ:RelDnDn1}
d^{(n+1)}_i = d^{(n)}_i - w_{i}^{(n+1)}d^{(n+1)}_{n+1}\qquad (i=0,1,\ldots,n),	
\end{equation}
where
\begin{equation}\label{EQ:wcoeff1}
w_{i}^{(n+1)} := \left<d^{(n)}_i,b_{n+1}\right>.
\end{equation}
\end{thm}

\noindent As a result of Theorem~\ref{T:RelDnDn1}, each dual function $d_i^{(n+1)}$ $(i=0,1,\ldots,n)$ depends on
$d_i^{(n)}$ which is known, and on $d^{(n+1)}_{n+1}$ which must be computed. Note that
$\mbox{span}\,B_{n+1} = \mbox{span}\left(D_{n}\cup\left\{b_{n+1}\right\}\right)$. Therefore, we can write
\begin{equation}\label{EQ:dn1n1}
d_{n+1}^{(n+1)} = \sum_{h=0}^{n}c_h^{(n+1)}d_{h}^{(n)} + c_{n+1}^{(n+1)}b_{n+1},
\end{equation}
and solve the following system of linear equations:
\begin{equation*}
\left\{\begin{array}{ll}
\displaystyle 0 = \left<d^{(n+1)}_{n+1},b_{i}\right> = c_i^{(n+1)} + c_{n+1}^{(n+1)}v_{i}^{(n+1)}
\qquad (i=0,1,\ldots,n),\\[1ex]
\displaystyle 1 = \left<d^{(n+1)}_{n+1},b_{n+1}\right> = \sum_{h=0}^{n}c_h^{(n+1)}w_h^{(n+1)} + c_{n+1}^{(n+1)}v_{n+1}^{(n+1)},
\end{array}\right.
\end{equation*}
where
\begin{equation}\label{EQ:vcoeff}
v_{j}^{(n+1)} := \left<b_{n+1},b_j\right> \qquad (j=0,1,\ldots,n+1),
\end{equation}
for the coefficients $c_0^{(n+1)},c_1^{(n+1)},\ldots,c_{n+1}^{(n+1)}$. According to \cite[Section 2]{Woz14}, the solution is simple, namely
\begin{align}
&c_{n+1}^{(n+1)} = \left(v_{n+1}^{(n+1)} - \sum_{h=0}^{n}v_{h}^{(n+1)}w_{h}^{(n+1)}\right)^{-1},\label{EQ:ccoeff1}\\[1ex]
&c_h^{(n+1)} = -v_h^{(n+1)}c_{n+1}^{(n+1)} \qquad (h=0,1,\ldots,n).\label{EQ:ccoeff2}
\end{align}

\noindent  The above-described idea is summarized in the following algorithm. We use this algorithm to compute the dual basis for the basis~\eqref{E:Case1}
in case 1 of the $i$th subproblem $(i > 1)$ (see Section~\ref{Sec:Idea}).

\begin{alg}[\cite{Woz14}]\,[\textsf{$D_n \Longrightarrow D_{n+1}$}]\label{A:DnDn1}\\[1.5ex]
\noindent \texttt{Input}: $D_{n} = \left\{d_0^{(n)},d_1^{(n)},\ldots,d_{n}^{(n)}\right\}$, $B_{n+1} = \left\{b_0,b_1,\ldots,b_{n+1}\right\}$\\[1ex]
\texttt{Output}: $D_{n+1} = \left\{d_0^{(n+1)},d_1^{(n+1)},\ldots,d_{n+1}^{(n+1)}\right\}$
\begin{description}
\itemsep2pt
\item[\texttt{Step 1}.] Compute $w_{i}^{(n+1)}$ $(i=0,1,\ldots,n)$ by~\eqref{EQ:wcoeff1}.
\item[\texttt{Step 2}.] Compute $v_{j}^{(n+1)}$ $(j=0,1,\ldots,n+1)$ by~\eqref{EQ:vcoeff}.
\item[\texttt{Step 3}.] Compute $c^{(n+1)}_{n+1}$ by~\eqref{EQ:ccoeff1}.
\item[\texttt{Step 4}.] Compute $c^{(n+1)}_{h}$ $(h=0,1,\ldots,n)$ by~\eqref{EQ:ccoeff2}.
\item[\texttt{Step 5}.] Compute $d^{(n+1)}_{n+1}$ by~\eqref{EQ:dn1n1}.
\item[\texttt{Step 6}.] Compute $d_i^{(n+1)}$ $(i=0,1,\ldots,n)$ by~\eqref{EQ:RelDnDn1}.
\item[\texttt{Step 7}.] Return the dual basis $\left\{d_0^{(n+1)},d_1^{(n+1)},\ldots,d_{n+1}^{(n+1)}\right\}$.
\end{description}
\end{alg}

The next algorithm computes a sequence of dual bases $D_0,D_1,\ldots,D_L$.
We use this algorithm to compute from scratch the dual basis for the basis~\eqref{E:First}
in the first iteration of the algorithm of solving Problem~\ref{P:Problem} (see Section~\ref{Sec:Idea}).
Since we are looking for the dual basis for the full basis~\eqref{E:First},
in our case only the last dual basis $D_L$ is needed.

\begin{alg}[\cite{Woz14}]\,[\textsf{Construction of dual bases $D_0,D_1,\ldots,D_L$}]\label{A:DN}\\[1.5ex]
\noindent \texttt{Input}: $B_{n} = \left\{b_0,b_1,\ldots,b_{n}\right\}$ $(n=0,1,\ldots,L)$\\[1ex]
\texttt{Output}: $D_{n} = \left\{d_0^{(n)},d_1^{(n)},\ldots,d_{n}^{(n)}\right\}$ $(n=0,1,\ldots,L)$
\begin{description}
\itemsep2pt
\item[\texttt{Step 1}.] Set $D_0 := \left\{\left<b_0,b_0\right>^{-1}b_0\right\}$.
\item[\texttt{Step 2}.] Compute $D_n$ $(n=1,2,\ldots,L)$ using Algorithm~\ref{A:DnDn1}.
\item[\texttt{Step 3}.] Return the dual bases $D_{0},D_{1},\ldots,D_{L}$.
\end{description}
\end{alg}

\begin{remark}[\cite{Woz14}]\label{R:Add}\normalfont
Let $f_n^\ast \in \mathcal{B}_n$ be the best least squares approximation of a function $g$ in the space $\mathcal{B}_n$, i.e.,
\begin{equation}\label{EQ:LSA}
\|g-f_n^\ast\|_2=\min_{f_n\in\mathcal{B}_n}\|g-f_n\|_2,
\end{equation}
where
\begin{equation}\label{EQ:LSA2}
f_n^\ast = \sum_{i=0}^n e_i^{(n)}b_i
\end{equation}
with $e_i^{(n)} := \left<g,d_i^{(n)}\right>$ for $i=0,1,\ldots,n$ (see Fact~\ref{T:DualLS}).
Suppose that we know the coefficients $e_i^{(n)}$ $(i=0,1,\ldots,n)$ and our goal is to compute the optimal element
$f_{n+1}^\ast \in \mathcal{B}_{n+1}$ for the same function $g$. Then Fact~\ref{T:DualLS}, along with the formulas~\eqref{EQ:dn1n1}
and~\eqref{EQ:RelDnDn1}, yields the following relations:
\begin{equation*}
\begin{array}{l}
\displaystyle e^{(n+1)}_{n+1} = \sum_{h = 0}^{n}c^{(n+1)}_he_h^{(n)}+c^{(n+1)}_{n+1}\left<g,b_{n+1}\right>,\\[3ex]
\displaystyle e^{(n+1)}_{i} =  e^{(n)}_{i} -w_i^{(n+1)}e^{(n+1)}_{n+1} \qquad (i=0,1,\ldots,n).
\end{array}
\end{equation*}
\end{remark}

\section{A new efficient method of modification of dual bases}\label{Sec:Prop}

In this section, we prove that for the given dual basis $D_{n+1} = \left\{d_0^{(n+1)},d_1^{(n+1)},\ldots,d_{n+1}^{(n+1)}\right\}$, it is possible to compute efficiently the dual basis $D_n = \left\{d_0^{(n)},d_1^{(n)},\ldots,d_{n}^{(n)}\right\}$. Notice that this is exactly the result that we need in case 2 of the $i$th subproblem $(i > 1)$ (see Section~\ref{Sec:Idea}).

\begin{remark}\normalfont
Observe that we cannot directly reverse the process from Section~\ref{Sec:Dual} using the formulas from Theorem~\ref{T:RelDnDn1}.
Clearly, we can rewrite the formula~\eqref{EQ:RelDnDn1}. However, we cannot combine~\eqref{EQ:RelDnDn1} with~\eqref{EQ:wcoeff1} because the searched coefficient $w_{i}^{(n+1)}$ in~\eqref{EQ:wcoeff1} \textbf{depends on the searched dual function} $d^{(n)}_i$. Therefore, the main goal of this section is to find a different formula for the coefficients $w_{i}^{(n+1)}$ $(i=0,1,\ldots,n)$ that \textbf{does not depend on the searched dual functions} $d^{(n)}_i$.
\end{remark}

In order to prove the main result, we will need the following lemma.

\begin{lem}\label{L:DualProp}
The following identity holds:
\begin{equation}\label{EQ:Lem}
\left< d_i^{(n)}, d_{n+1}^{(n+1)} \right> = 0 \qquad (i=0,1,\ldots,n).
\end{equation}
\end{lem}

\begin{proof}
We use Fact~\ref{T:DualRepr} to represent each dual function $d_i^{(n)}$ as a linear combination of the elements $b_0, b_1,\ldots,b_n$,
\begin{equation*}
d_i^{(n)} = \sum_{j=0}^n \left< d_i^{(n)}, d_j^{(n)} \right>b_j.
\end{equation*}
Consequently, we have
\begin{align*}
\left< d_i^{(n)}, d_{n+1}^{(n+1)} \right> &= \left< \sum_{j=0}^n \left< d_i^{(n)}, d_j^{(n)} \right>b_j, d_{n+1}^{(n+1)} \right>\\[1ex]
&= \sum_{j=0}^n \left < d_i^{(n)}, d_j^{(n)} \right> \left< b_j, d_{n+1}^{(n+1)} \right>= 0
\end{align*}
since $\left< b_j, d_{n+1}^{(n+1)} \right> = 0$ for $j=0,1,\ldots,n$.
\end{proof}

\begin{thm}\label{T:RelDn1Dn}
The connection between the dual functions from $D_n$ and $D_{n+1}$ is as follows:
\begin{equation}\label{EQ:RelDnDn12}
d_i^{(n)} = d_i^{(n+1)} + w_i^{(n+1)}d_{n+1}^{(n+1)}\qquad (i=0,1,\ldots,n),
\end{equation}
where
\begin{equation}\label{EQ:wcoeff2}
w_i^{(n+1)} := - \frac{\left< d_i^{(n+1)}, d_{n+1}^{(n+1)} \right>}{\left< d_{n+1}^{(n+1)}, d_{n+1}^{(n+1)} \right>}
\end{equation}
(cf.~Theorem~\ref{T:RelDnDn1}).
\end{thm}

\begin{proof}
Obviously, the relation~\eqref{EQ:RelDnDn12} follows from~\eqref{EQ:RelDnDn1}.
Now, we substitute~\eqref{EQ:RelDnDn12} into the equation~\eqref{EQ:Lem} and obtain
\begin{align*}
0 &= \left< d_i^{(n+1)} + w_i^{(n+1)}d_{n+1}^{(n+1)}, d_{n+1}^{(n+1)} \right>\\[1ex]
&= \left< d_i^{(n+1)}, d_{n+1}^{(n+1)} \right> + w_i^{(n+1)}\left< d_{n+1}^{(n+1)}, d_{n+1}^{(n+1)} \right>.
\end{align*}
Hence, the formula~\eqref{EQ:wcoeff2} follows.
\end{proof}

\begin{remark}\normalfont
In contrast to~\eqref{EQ:wcoeff1}, the formula~\eqref{EQ:wcoeff2} is \textbf{independent} of $d_i^{(n)}$ $(i=0,1,\ldots,n)$.
Therefore, it can be used to compute the dual basis $D_n$, under the assumption that the dual basis $D_{n+1}$ is given.
\end{remark}

We use the following algorithm to compute the dual basis for the basis~\eqref{E:Case2} in case 2 of the $i$th subproblem $(i > 1)$
(see Section~\ref{Sec:Idea}).

\begin{alg}\,[\textsf{$D_{n+1} \Longrightarrow D_{n}$}]\label{A:Dn1Dn}\\[1.5ex]
\noindent \texttt{Input}: $D_{n+1} = \left\{d_0^{(n+1)},d_1^{(n+1)},\ldots,d_{n+1}^{(n+1)}\right\}$\\[1ex]
\texttt{Output}: $D_{n} = \left\{d_0^{(n)},d_1^{(n)},\ldots,d_{n}^{(n)}\right\}$
\begin{description}
\itemsep2pt
\item[\texttt{Step 1}.] For $i=0,1,\ldots,n$,
\begin{itemize}
\setlength{\itemindent}{-0.5cm}
\item[\texttt{\normalfont{(i)}}] compute $w_i^{(n+1)}$ by~\eqref{EQ:wcoeff2};
\item[\texttt{\normalfont{(ii)}}] compute $d_i^{(n)}$ by~\eqref{EQ:RelDnDn12}.
\end{itemize}
\item[\texttt{Step 2}.] Return the dual basis $\left\{d_0^{(n)},d_1^{(n)},\ldots,d_{n}^{(n)}\right\}$.
\end{description}
\end{alg}

\begin{remark}\normalfont
Suppose that a dual basis of a certain space is well-known or was computed earlier. In CAGD, we often look for an optimal element (in the least squares sense) which is \textit{constrained}, e.g., by some continuity conditions. As a result, we need a dual basis of a specific subspace of the well-known space. Algorithm~\ref{A:Dn1Dn} can be particularly useful in those situations. The idea from Section~\ref{Sec:Idea} is only an example of its application.
\end{remark}

\begin{remark}\label{R:Rem}\normalfont
Let $f_{n+1}^\ast \in \mathcal{B}_{n+1}$ be the best least squares approximation of a function $g$ in the space $\mathcal{B}_{n+1}$, i.e.,
$$
\|g-f_{n+1}^\ast\|_2=\min_{f_{n+1}\in\mathcal{B}_{n+1}}\|g-f_{n+1}\|_2,
$$
where
$$
f_{n+1}^\ast = \sum_{j=0}^{n+1} e_j^{(n+1)}b_j
$$
with $e_j^{(n+1)} := \left<g,d_j^{(n+1)}\right>$ for $j=0,1,\ldots,n+1$ (see Fact~\ref{T:DualLS}).
Suppose that we know the coefficients $e_j^{(n+1)}$ $(j=0,1,\ldots,n+1)$ and our goal is to compute the optimal element
$f_{n}^\ast \in \mathcal{B}_{n}$ for the same function $g$ (see~\eqref{EQ:LSA} and \eqref{EQ:LSA2}). Then Fact~\ref{T:DualLS}, along with the formula~\eqref{EQ:RelDnDn12},
yields the following relation:
\begin{equation*}
e^{(n)}_{i} =  e^{(n+1)}_{i} + w_i^{(n+1)}e^{(n+1)}_{n+1} \qquad (i=0,1,\ldots,n)
\end{equation*}
(cf.~Remark~\ref{R:Add}).
\end{remark}

\begin{remark}\normalfont
Recall that in both cases of the $i$th subproblem $(i > 1)$ in Section~\ref{Sec:Idea}, $\varphi_i$ only \textit{slightly} differs from $\varphi_{i-1}$
(see \eqref{E:phi1} and \eqref{E:phi2}). Similarly as in Remarks~\ref{R:Add} and~\ref{R:Rem}, one can obtain formulas connecting the coefficients of the new optimal element $\psi_i^{\ast}$ and the previous one $\psi_{i-1}^{\ast}$. In case 2, the formula is simple, efficient and worth considering. For details, see our implementation in $\mbox{Maple}{\small \texttrademark}13$ available on the webpage \url{http://www.ii.uni.wroc.pl/~pgo/papers.html}.
\end{remark}

\section{Examples}\label{Sec:Ex}

In this section, we consider the problem of separate degree reduction of sixteen segments of the composite B\'ezier curve ``Octopus''
(see Figure~\ref{figure1a}). Notice that the original control points are located very close to the plot of the composite curve (see Figure~\ref{figure1b}).

The results have been obtained on a computer with \texttt{Intel Core i5-3337U 1.8GHz} processor and \texttt{8GB} of \texttt{RAM}, using $16$-digit arithmetic. $\mbox{Maple}{\small \texttrademark}13$ worksheet containing programs and tests can be found at \url{http://www.ii.uni.wroc.pl/~pgo/papers.html}. Text file with the control points of the composite B\'ezier curve ``Octopus''  is available as well.

We apply the algorithms independently to every segment of the composite curve. In each case, we use the sequence $T = \{t_k\}^N_{k=0}$  of equally spaced points for the least squares distance~\eqref{EQ:min}, i.e., we set $t_k := k/N$ $(k = 0,1,\ldots,N)$.
In Table~\ref{tab:table1}, we give the parameters, least squares errors $E$ (see~\eqref{EQ:min}) and maximum errors
$$
E_{\infty} := \max_{t \in S_M} ||P_n(t) - R_m(t)|| \approx \max_{0 \leq t \leq 1} ||P_n(t) - R_m(t)||,
$$
where $S_M := \left\{0, 1/M, 2/M,\ldots, 1\right\}$ with $M := 500$.

As a result of the traditional degree reduction (see Remark~\ref{R:Trad}), we obtain the composite curve with the control points shown in Figure~\ref{figure1c}. Clearly, some of the control points are located far away from the plot of the curve (cf.~Figure~\ref{figure1b}).

Next, to perform degree reduction with box constraints (see Problem~\ref{P:Problem}), we use the new idea from Section~\ref{Sec:Idea} combined with the algorithms from Sections~\ref{Sec:Dual} and~\ref{Sec:Prop}. For each resulting B\'{e}zier curve, the box constraints~\eqref{EQ:box} were chosen so that the searched control points are placed inside the rectangular area bounded by the outermost control points of the original corresponding B\'ezier curve. More precisely, we set
\begin{equation*}
l_z := \min_{0 \leq k \leq n}p_k^z, \qquad u_z := \max_{0 \leq k \leq n}p_k^z \qquad (z = x, y).
\end{equation*}
The resulting composite curve with its control points is illustrated in Figure~\ref{figure1d}. This time, the location of the control points is much more satisfying (cf.~Figure~\ref{figure1c}). However, because of the additional restrictions~\eqref{EQ:box}, the larger errors are unavoidable (see Table~\ref{tab:table1}).

In Table~\ref{tab:table2}, we give the comparison of total running times between the box-constrained degree reduction
from \cite[Section 5 and Appendix]{Gos15} and the new method. Clearly, the approach presented in this paper is approximately
two times faster than the older one.

\begin{table}[H]
\captionsetup{margin=0pt, font={scriptsize}}
\centering
\ra{1.5}
\scalebox{0.8}{
\begin{tabular}{@{}clcccccccccccc@{}}
 \toprule \multicolumn{8}{c}{Input data} & \phantom{a} &
\multicolumn{2}{c}{Traditional degree reduction} & \phantom{a} &
\multicolumn{2}{c}{Problem~\ref{P:Problem}}\\
 \cmidrule{1-8} \cmidrule{10-11}\cmidrule{13-14} & Curves & & $n$ & $m$ & $N$ & $\alpha$ & $\beta$ & & $E$ & $E_{\infty}$ & & $E$ & $E_{\infty}$
\\ \midrule
\multirow{16}{*}{\begin{sideways}Octopus\end{sideways}}
& Head: left side & & $9$ & $7$ & $20$ & $2$ & $1$ &  & $5.07e{-}4$ & $2.30e{-}4$ &  & $5.72e{-}3$ & $2.32e{-}3$\\

& Head: right side & & $9$  & $7$ & $20$ & $1$ & $0$ &  & $9.05e{-}5$ & $5.12e{-}5$ & & $2.29e{-}3$ & $8.86e{-}4$\\

& $1$st arm: part 1 & & $15$  & $9$ & $23$ & $0$ & $0$ &  & $2.65e{-}4$ & $1.58e{-}4$ & & $3.53e{-}3$ & $1.40e{-}3$\\

& $1$st arm: part 2 & & $17$  & $9$ & $28$ & $0$ & $1$ &  & $1.59e{-}3$ & $7.86e{-}4$ & & $1.87e{-}3$ & $8.92e{-}4$\\

& $2$nd arm: part 1 & & $14$ & $10$ & $26$ & $1$ & $0$ &  & $1.62e{-}4$ & $9.51e{-}5$ & & $1.32e{-}2$ & $4.41e{-}3$\\

& $2$nd arm: part 2 & & $14$ & $10$ & $26$ & $0$ & $1$ &  & $7.35e{-}5$ & $3.54e{-}5$ & & $4.44e{-}3$ & $2.44e{-}3$\\

& $3$rd arm: part 1 & & $13$ & $7$ & $25$ & $1$ & $1$ &  & $2.50e{-}3$ & $9.56e{-}4$  & & $2.62e{-}2$ & $9.03e{-}3$\\

& $3$rd arm: part 2 & & $11$ & $7$ & $23$ & $1$ & $0$ &  & $8.73e{-}4$ & $5.13e{-}4$ & & $4.05e{-}3$ & $1.62e{-}3$\\

& $4$th arm & & $11$ & $7$ & $23$ & $0$ & $0$ &  & $2.78e{-}3$ & $1.47e{-}3$ & & $2.61e{-}2$ & $1.01e{-}2$\\

& $5$th arm & & $18$ & $11$ & $23$ & $0$ & $0$ &  & $1.51e{-}4$ & $2.91e{-}4$ & & $6.14e{-}3$ & $2.41e{-}3$\\

& $6$th arm: part 1 & & $12$ & $7$ & $28$ & $0$ & $0$ &  & $1.44e{-}3$ & $6.37e{-}4$ & & $6.97e{-}3$ & $2.71e{-}3$\\

& $6$th arm: part 2 & & $17$ & $9$ & $29$ & $0$ & $2$ &  & $8.48e{-}4$ & $5.16e{-}4$ & & $1.39e{-}2$ & $4.26e{-}3$\\

& $7$th arm: part 1 & & $15$ & $9$ & $29$ & $2$ & $0$ &  & $9.86e{-}4$ & $3.60e{-}4$ & & $4.39e{-}3$ & $1.81e{-}3$\\

& $7$th arm: part 2 & & $11$ & $7$ & $25$ & $0$ & $2$ &  & $8.95e{-}4$ & $3.54e{-}4$ & & $8.95e{-}4$ & $3.54e{-}4$\\

& $8$th arm: part 1 & & $16$ & $9$ & $29$ & $2$ & $0$ &  & $1.21e{-}3$ & $4.44e{-}4$& & $1.38e{-}2$& $4.66e{-}3$\\

& $8$th arm: part 2 & & $9$ & $6$ & $18$ & $0$ & $2$ &  & $2.29e{-}3$ & $1.05e{-}3$& & $2.29e{-}3$& $1.05e{-}3$\\
\hline
\end{tabular}}
\caption{The results of separate degree reduction of segments of the composite B\'{e}zier curve ``Octopus''.}
\label{tab:table1}
\end{table}

\begin{table}[H]
\captionsetup{margin=0pt, font={scriptsize}}
\centering
\ra{1.5}
\scalebox{0.8}{
\begin{tabular}{@{}cccc@{}}
\toprule & \phantom{a} & Without the use of dual bases (\cite{Gos15}) & With the use of dual bases (present paper)\\ \midrule
Running times [s] & & $2.436$ & $1.249$\\
\hline
\end{tabular}}
\caption{Total running times of separate box-constrained degree reduction of segments of the composite B\'{e}zier curve ``Octopus''. For the parameters, see Table~\ref{tab:table1}.}
\label{tab:table2}
\end{table}

\begin{figure}[H]
\captionsetup{margin=0pt, font={scriptsize}}
\begin{center}
\setlength{\tabcolsep}{0mm}
\begin{tabular}{c}
\subfloat[]{\label{figure1a}\includegraphics[width=0.45\textwidth]{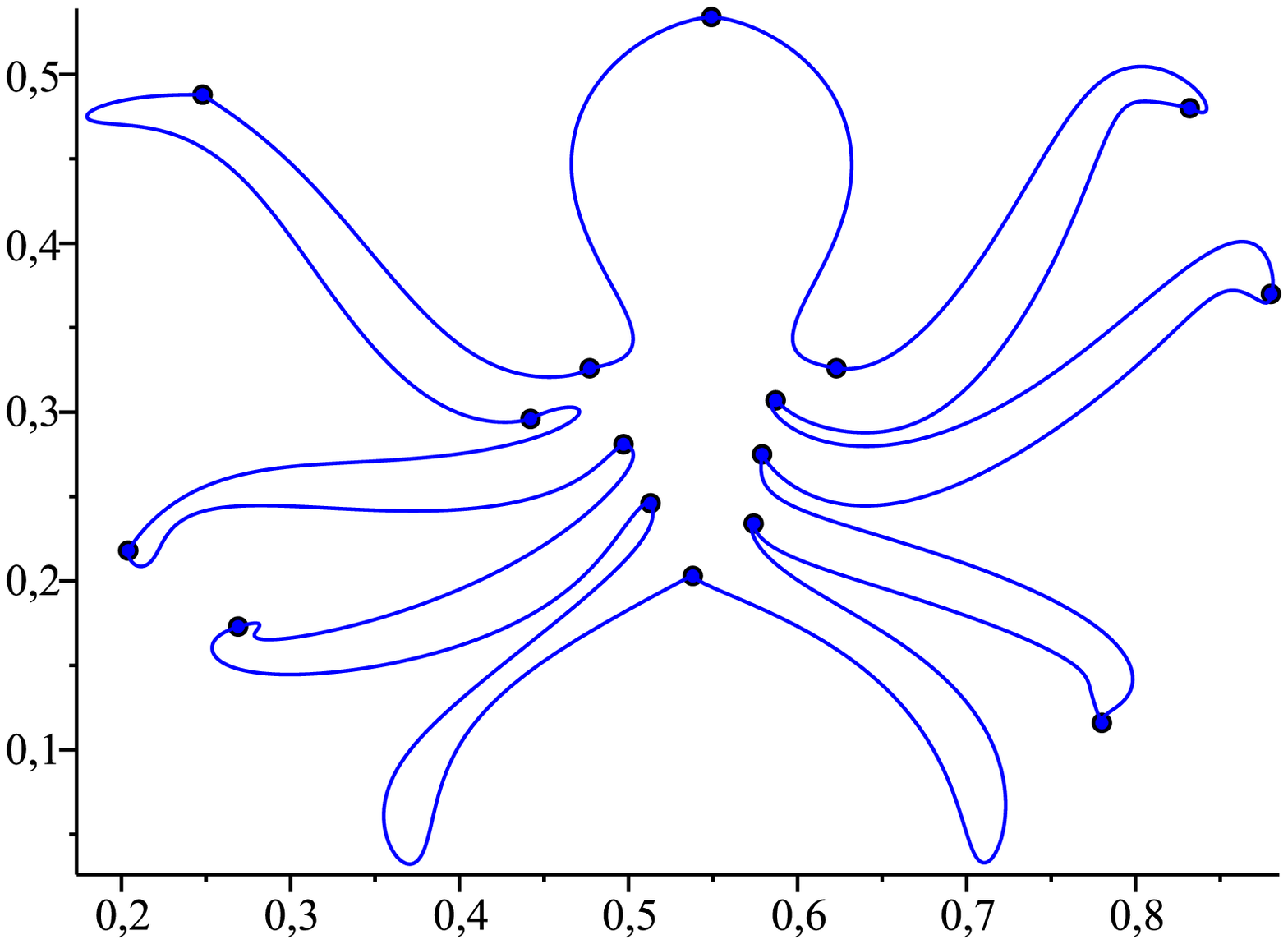}}
\subfloat[]{\label{figure1b}\includegraphics[width=0.45\textwidth]{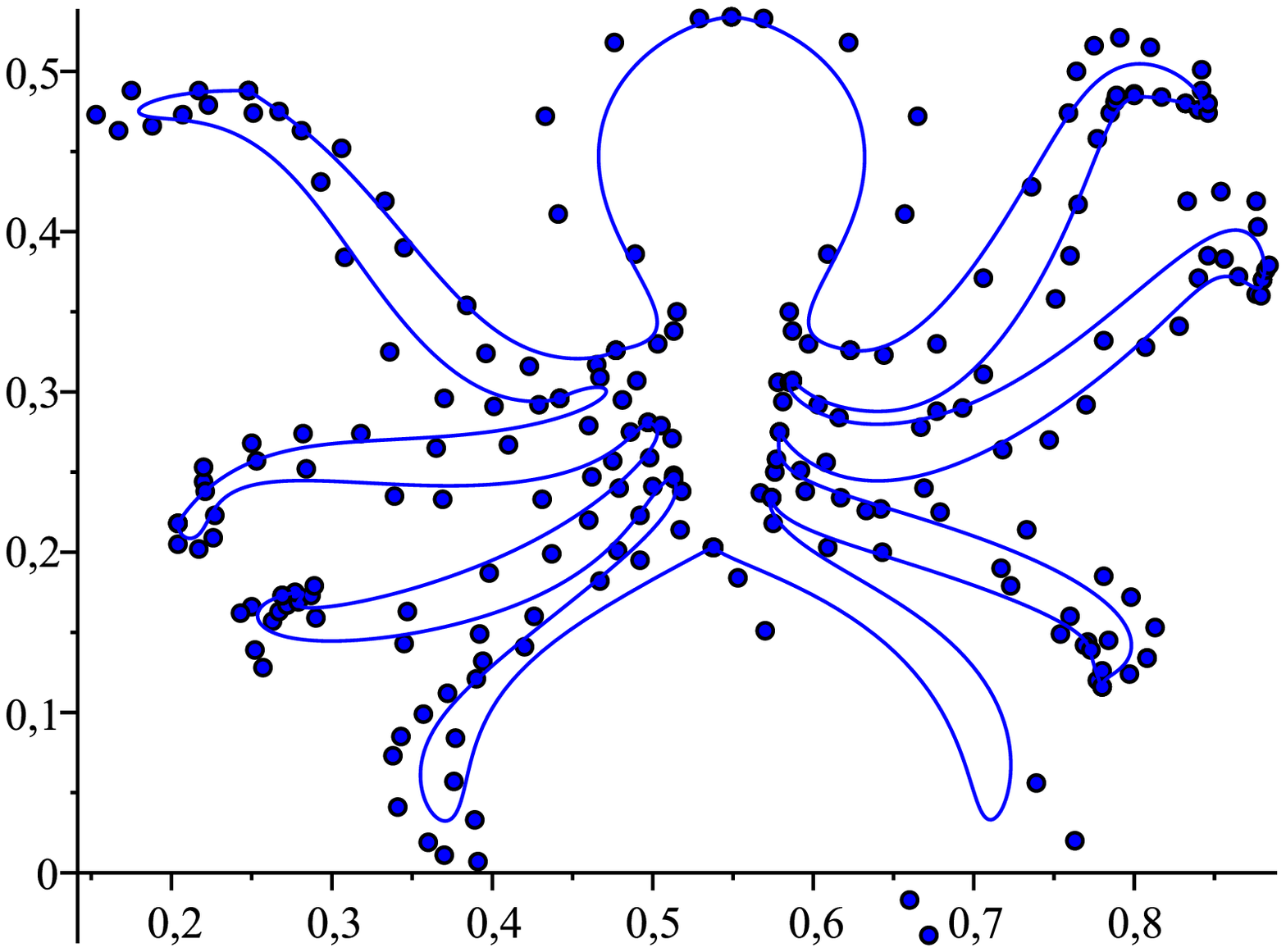}}\\
\subfloat[]{\label{figure1c}\includegraphics[width=0.45\textwidth]{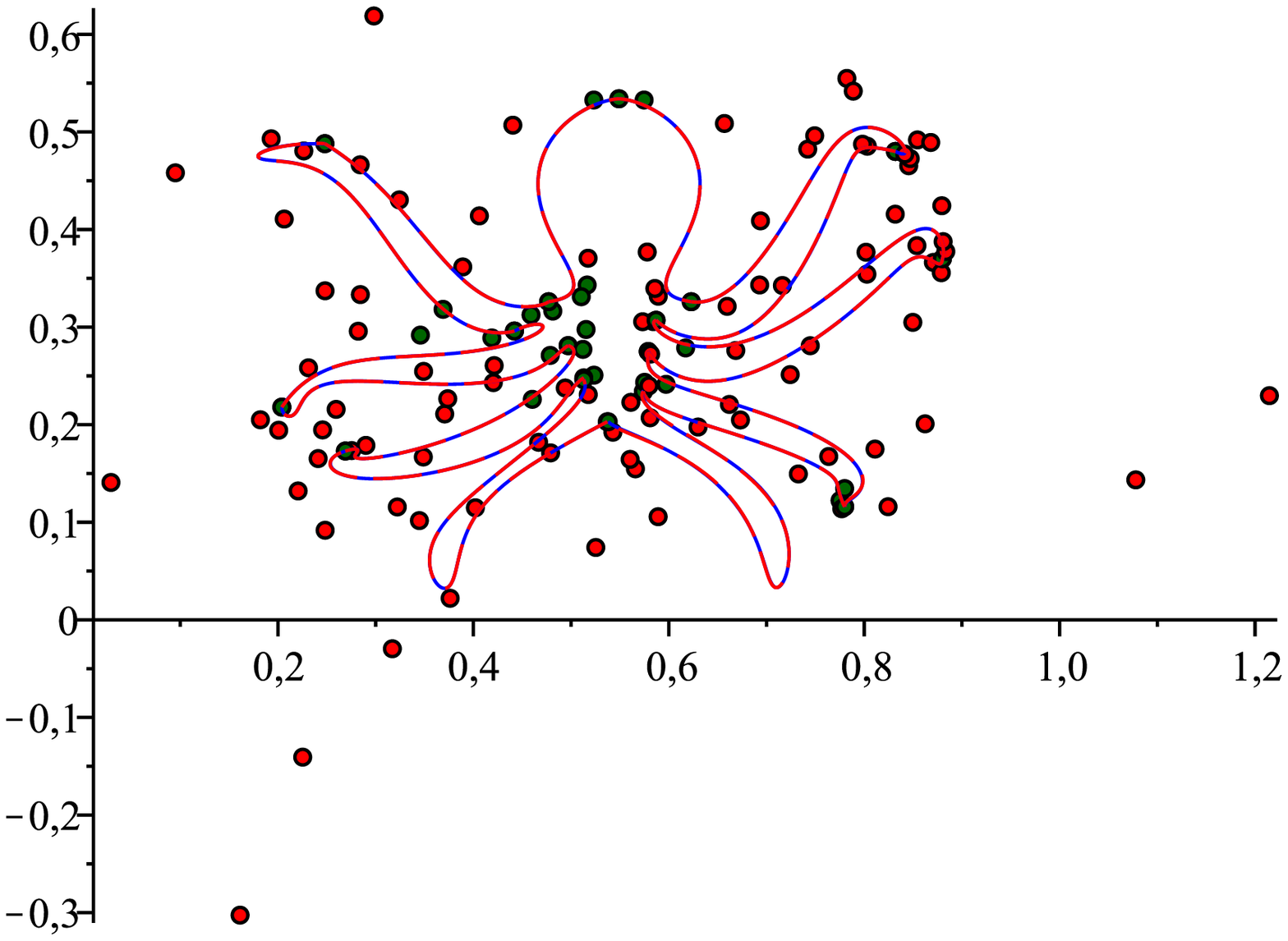}}
\subfloat[]{\label{figure1d}\includegraphics[width=0.45\textwidth]{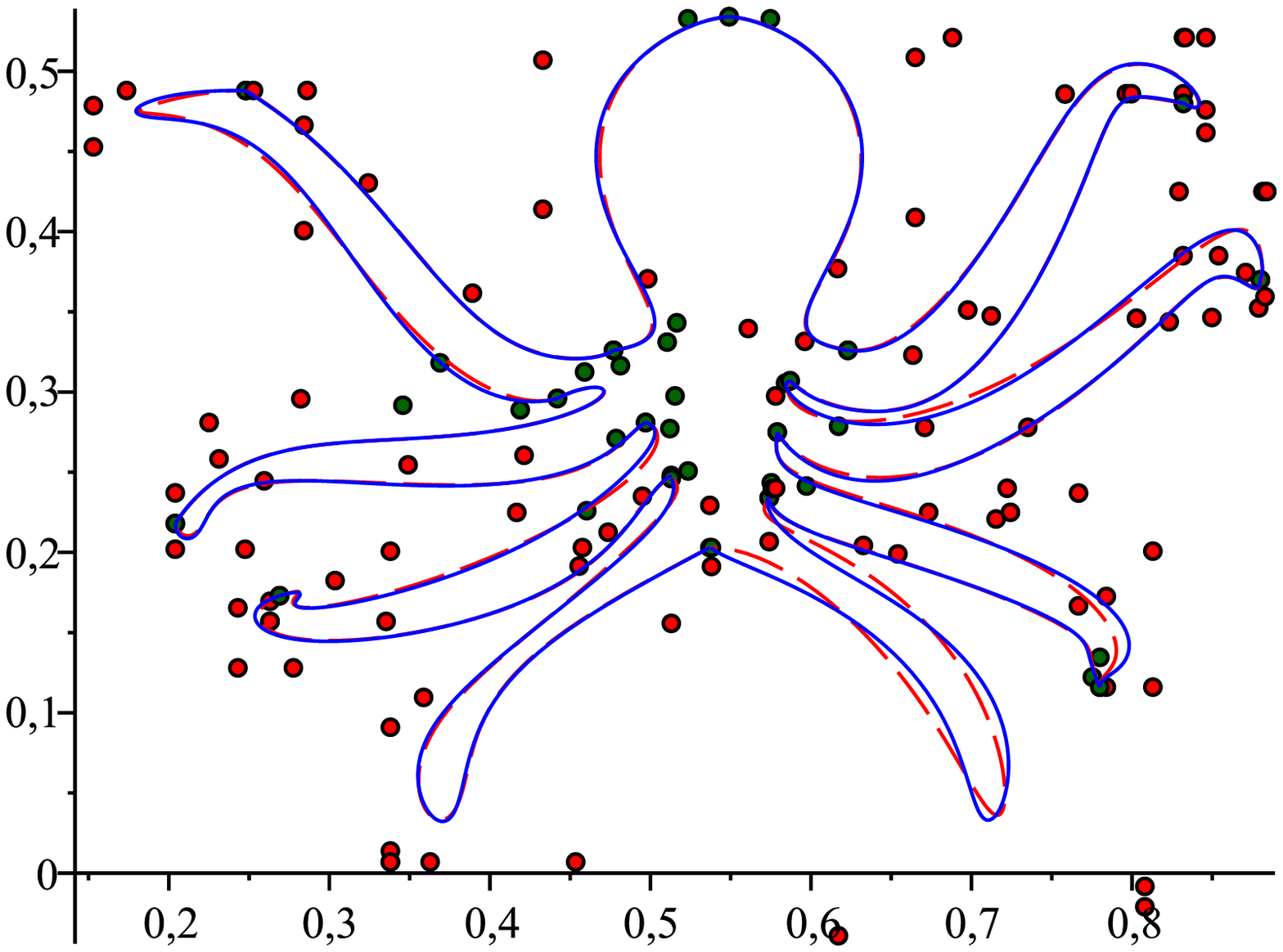}}
\end{tabular}
    \caption{Figure (a) shows the original composite B\'ezier curve ``Octopus''. Figure (b) presents the same composite curve as Figure (a) but with its control points. Figures (c) and (d) illustrate separate degree reduction of segments of the original composite curve (blue solid line) to degree reduced composite curve (red dashed line with red and green control points), where each segment is (c) optimal solution of the traditional degree reduction, (d) optimal solution of Problem~\ref{P:Problem}. The control points which are constrained by the continuity conditions~\eqref{EQ:cont} are green, while the other ones are red and restricted by~\eqref{EQ:box}. Parameters are specified in Table~\ref{tab:table1}.}
\end{center}
\end{figure}

\section{Conclusions}\label{Sec:Conc}
In this paper, we have improved the iterative algorithm of degree reduction of B\'ezier curves with box constraints from \cite{Gos15}.
In order to achieve our goal, we have combined some old and new results on dual bases with the observation that the subproblems in consecutive iterations of the original algorithm are related. The experiments have shown that the new approach is approximately two times faster than the previous one from \cite{Gos15}. Furthermore, in each iteration, we have avoided solving a system of normal equations which can be associated with matrix inversion.


\bibliographystyle{plain}


\end{document}